\documentclass{amsart}
\pdfoutput=1

\usepackage{microtype}
\usepackage{booktabs}
\usepackage[pdftex,colorlinks,citecolor=black,linkcolor=black,urlcolor=black,bookmarks=false]{hyperref}
\newcommand{\arXiv}[1]{arXiv:\href{http://arXiv.org/abs/#1}{#1}}

\usepackage{doi}
\usepackage{tikz}
\usepackage{array}
\newcolumntype{A}{ >{$} r <{$} @{} >{${}} l <{$} }

\newtheorem{theorem}{Theorem}[section]

\newtheorem{conjecture}[theorem]{Conjecture}

\theoremstyle{definition}

\numberwithin{equation}{section}
\numberwithin{table}{section}
\numberwithin{figure}{section}

\newcommand{\R}{\mathbb{R}}

\newcommand{\F}{\mathbb{F}}

\title{Variations on five-dimensional sphere packings}

\author{Henry Cohn}
\address{Microsoft Research New England, One Memorial Drive, Cambridge, MA
02142, USA}
\curraddr{Department of Mathematics, Massachusetts Institute of Technology, Cambridge, MA}
\email{cohn@mit.edu}

\author{Isaac Rajagopal}
\address{Department of Mathematics, Massachusetts Institute of Technology, Cambridge, MA}
\email{isaacraj@mit.edu}

\begin{document}

\begin{abstract}
We analyze Sz\"oll\H{o}si's recent construction of a conjecturally optimal five-dimensional kissing configuration and produce a new such configuration, the fourth to be discovered. We construct five-dimensional sphere packings from these configurations, which augment Conway and Sloane's list of conjecturally optimal packings. We also construct a new kissing configuration in nine dimensions. None of these constructions improves on the known records, but they provide geometrically distinct constructions achieving these records.
\end{abstract}

\maketitle

\section{Introduction}

The sphere packing and kissing problems are two closely related problems in discrete geometry: the former asks how densely congruent spheres can be arranged in $\R^n$ with disjoint interiors, while the latter asks how many such spheres can be tangent to one central sphere (of the same size). Like many packing and coding problems, the answers are known only in relatively low dimensions, while the problems become increasingly mysterious in higher dimensions. The sphere packing problem has been solved in one through three, eight, and twenty-four dimensions \cite{thue1892,hales2005,flyspeck2017,viazovska2017,ckmrv2017}, while the kissing problem has been solved in one through four, eight, and twenty-four dimensions \cite{sw1953,musin2008,os1979,levenshtein1979}.

In dimensions where no proof is known, it is natural to wonder how confidently we can identify the optimal packings. Conway and Sloane \cite{cs1995} took a major step in this direction by classifying how packings can be obtained by stacking layers, and they gave a conjectural list of optimal sphere packings in dimensions up to nine. (Above nine dimensions, this approach fails, which illustrates the increasing complexity of the sphere packing problem.) Their list does not contain every possible densest packing, because one can always make local perturbations without changing the global density, but they conjectured that it contains all the ``tight'' packings. Kuperberg \cite{kuperberg2000} disproved their specific formulation of tightness and suggested a replacement in terms of weak recurrence (Conjecture~15 in \cite{kuperberg2000}). With this amended formulation, Conway and Sloane's list of packings seemed plausibly complete, as did the analogous classification of kissing configurations by Cohn, Jiao, Kumar, and Torquato \cite{cjkt2011}. However, in 2023 Sz\"oll\H{o}si \cite{szollosi2023} found a new five-dimensional kissing configuration, which could be extended to a new five-dimensional sphere packing (see the acknowledgements in \cite{szollosi2023}). This example showed that even in as few as five dimensions, our understanding of optimal sphere packings was incomplete.

In this paper, we put Sz\"oll\H{o}si's construction in a broader context. We construct another five-dimensional kissing configuration, which brings the list of known examples to four, and we extend these kissing configurations to a family of five-dimensional sphere packings. It is conceivable that our list completes Conway and Sloane's list in five dimensions, but it is difficult to rule out further surprises. The two most noteworthy sphere packings in the new family are uniform, which means their symmetry groups act transitively on the spheres. Only four such packings had previously been analyzed, namely the $D_5$ root lattice and three packings constructed by Leech \cite{leech1967}. We prove that these six uniform packings are the only uniform packings that contain the known kissing configurations locally. 

The unique uniform packing with Sz\"oll\H{o}si's kissing configuration was previously discovered by Andreanov and Kallus \cite{ak2020} as $J_{5,5}$ in their classification of dense $2$-periodic packings (i.e., unions of two translates of a lattice). In fact, their discovery predates Sz\"oll\H{o}si's, although they did not recognize that the kissing configuration was novel. The remaining uniform packing, with our new kissing configuration, is $4$-periodic. Viewing these uniform packings as $m$-periodic with $m$ small makes it easier to analyze their symmetry groups.

We have had no luck applying our techniques to six or seven dimensions. We expect that further kissing configurations and packings remain to be discovered in these dimensions, but doing so may require a new idea. Instead, we construct a new example in nine dimensions, which again does not set a new record but provides a geometrically distinct way to achieve the known record.

All of these constructions can be understood through modifying layers in packings. We will describe kissing configurations and packings in terms of the locations of the sphere centers. In other words, a kissing configuration consists of an arrangement of points on the surface of a sphere such that each pair of distinct points forms an angle of at least $\pi/3$ with the center of the sphere, and a sphere packing consists of an arrangement of points in Euclidean space with a certain minimal distance (namely, twice the sphere radius). Our altered configurations will remove the points on a hyperplane and replace them with modified points. The simplest case is the three-dimensional kissing configuration of the face-centered cubic lattice, in which it is well known that one can replace a triangle of points as shown in Figure~\ref{figure3d}.

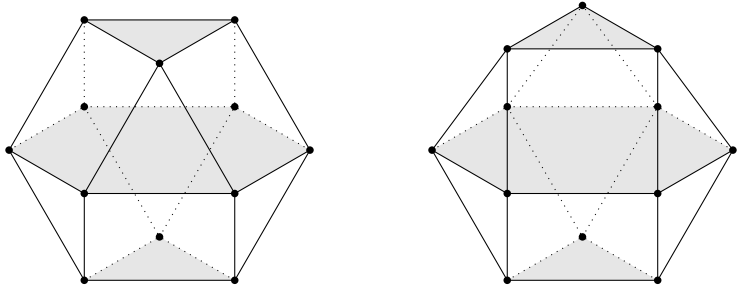
\begin{figure}
\centering
\begin{tikzpicture}
\fill[black!10!white] (1,1.7320508)--(-1,1.7320508)--(0,1.1547005)--cycle;
\fill[black!10!white] (-1,-1.7320508)--(1,-1.7320508)--(0,-1.1547005)--cycle;
\fill[black!10!white] (-2,0)--(-1,0.57735027)--(1,0.57735027)--(2,0)--(1,-0.57735027)--(-1,-0.57735027)--cycle;
\fill (1,1.7320508) circle (0.05);
\fill (-1,1.7320508) circle (0.05);
\fill (-2,0) circle (0.05);
\fill (-1,-1.7320508) circle (0.05);
\fill (1,-1.7320508) circle (0.05);
\fill (2,0) circle (0.05);
\fill (-1,0.57735027) circle (0.05);
\fill (0,-1.1547005) circle (0.05);
\fill (1,0.57735027) circle (0.05);
\fill (1,-0.57735027) circle (0.05);
\fill (0,1.1547005) circle (0.05);
\fill (-1,-0.57735027) circle (0.05);
\draw (1,1.7320508)--(-1,1.7320508);
\draw (1,1.7320508)--(2,0);
\draw[dotted] (1,1.7320508)--(1,0.57735027);
\draw (1,1.7320508)--(0,1.1547005);
\draw (-1,1.7320508)--(-2,0);
\draw[dotted] (-1,1.7320508)--(-1,0.57735027);
\draw (-1,1.7320508)--(0,1.1547005);
\draw (-2,0)--(-1,-1.7320508);
\draw[dotted] (-2,0)--(-1,0.57735027);
\draw (-2,0)--(-1,-0.57735027);
\draw (-1,-1.7320508)--(1,-1.7320508);
\draw[dotted] (-1,-1.7320508)--(0,-1.1547005);
\draw (-1,-1.7320508)--(-1,-0.57735027);
\draw (1,-1.7320508)--(2,0);
\draw[dotted] (1,-1.7320508)--(0,-1.1547005);
\draw (1,-1.7320508)--(1,-0.57735027);
\draw[dotted] (2,0)--(1,0.57735027);
\draw (2,0)--(1,-0.57735027);
\draw[dotted] (-1,0.57735027)--(0,-1.1547005);
\draw[dotted] (-1,0.57735027)--(1,0.57735027);
\draw[dotted] (0,-1.1547005)--(1,0.57735027);
\draw (1,-0.57735027)--(0,1.1547005);
\draw (1,-0.57735027)--(-1,-0.57735027);
\draw (0,1.1547005)--(-1,-0.57735027);
\end{tikzpicture}
\qquad \qquad
\begin{tikzpicture}
\fill[black!10!white] (-1,1.3471506)--(1,1.3471506)--(0,1.9245009)--cycle;
\fill[black!10!white] (-1,-1.7320508)--(1,-1.7320508)--(0,-1.1547005)--cycle;
\fill[black!10!white] (-2,0)--(-1,0.57735027)--(1,0.57735027)--(2,0)--(1,-0.57735027)--(-1,-0.57735027)--cycle;
\fill (-1,1.3471506) circle (0.05);
\fill (1,1.3471506) circle (0.05);
\fill (-2,0) circle (0.05);
\fill (-1,-1.7320508) circle (0.05);
\fill (1,-1.7320508) circle (0.05);
\fill (2,0) circle (0.05);
\fill (-1,0.57735027) circle (0.05);
\fill (0,-1.1547005) circle (0.05);
\fill (1,0.57735027) circle (0.05);
\fill (1,-0.57735027) circle (0.05);
\fill (0,1.9245009) circle (0.05);
\fill (-1,-0.57735027) circle (0.05);
\draw (-1,1.3471506)--(1,1.3471506);
\draw (-1,1.3471506)--(-2,0);
\draw (-1,1.3471506)--(0,1.9245009);
\draw (-1,1.3471506)--(-1,-0.57735027);
\draw (1,1.3471506)--(2,0);
\draw (1,1.3471506)--(1,-0.57735027);
\draw (1,1.3471506)--(0,1.9245009);
\draw (-2,0)--(-1,-1.7320508);
\draw[dotted] (-2,0)--(-1,0.57735027);
\draw (-2,0)--(-1,-0.57735027);
\draw (-1,-1.7320508)--(1,-1.7320508);
\draw[dotted] (-1,-1.7320508)--(0,-1.1547005);
\draw (-1,-1.7320508)--(-1,-0.57735027);
\draw (1,-1.7320508)--(2,0);
\draw[dotted] (1,-1.7320508)--(0,-1.1547005);
\draw (1,-1.7320508)--(1,-0.57735027);
\draw[dotted] (2,0)--(1,0.57735027);
\draw (2,0)--(1,-0.57735027);
\draw[dotted] (-1,0.57735027)--(0,-1.1547005);
\draw[dotted] (-1,0.57735027)--(1,0.57735027);
\draw[dotted] (-1,0.57735027)--(0,1.9245009);
\draw[dotted] (0,-1.1547005)--(1,0.57735027);
\draw[dotted] (1,0.57735027)--(0,1.9245009);
\draw (1,-0.57735027)--(-1,-0.57735027);
\end{tikzpicture}
\caption{The centers of the spheres in the face-centered cubic kissing arrangement on the left, and the centers of the spheres in the hexagonal close packing kissing arrangement on the right. To produce the arrangement on the right from the one on the left, the top layer is deleted and the bottom layer is reflected across the central layer (the hexagon).}\label{figure3d}
\end{figure}

In the remainder of this paper, we analyze the five-dimensional kissing configurations and construct our new configuration (Section~\ref{sec:5dkissing}), examine what goes wrong in six dimensions (Section~\ref{sec:6dkissing}), extend these kissing configurations to sphere packings (Section~\ref{sec:5dpacking}), and conclude with our construction in nine dimensions (Section~\ref{sec:9dkissing}). We also provide computer code for verifying certain assertions through DSpace@MIT at \url{https://hdl.handle.net/1721.1/157699}.

\section{Five-dimensional kissing configurations}
\label{sec:5dkissing}

The kissing number in five dimensions appears to be $40$, although the best upper bound that has been proved is~$44$ (from \cite{mv2010}). The first construction achieving~$40$ is implicit in Korkine and Zolotareff's 1873 paper \cite{kz1873}, where they constructed the $D_5$ root lattice. Its root system achieves a kissing number of~$40$ as the permutations of the points $(\pm 1, \pm 1, 0, 0, 0)$; these points form a kissing configuration because they each have squared norm~$2$ and the inner product between distinct points is always at most~$1$.

In 1967, Leech \cite{leech1967} constructed a different kissing configuration of the same size, not isometric to the $D_5$ root system. He split $D_5$ into four-dimensional layers based on the fifth coordinate: there are $8$ points with fifth coordinate~$1$, $24$~with~$0$, and $8$~with~$-1$. The top layer consists of the points
\[
(\overbrace{\pm 1, 0, 0, 0}^{\text{permute}},1),
\]
and Leech observed that replacing it with
\[
(\overbrace{\pm \tfrac{1}{2}, \pm \tfrac{1}{2}, \pm \tfrac{1}{2}, \pm \tfrac{1}{2}}^{\text{odd \# of $-$ signs}},1)
\]
yields a different kissing configuration $L_5$, which breaks the reflection symmetry across the central hyperplane with fifth coordinate~$0$. One could also use an even number of minus signs, or replace both the top and bottom layers, but these modifications are all isometric to $D_5$ or $L_5$. For comparison, such a construction cannot be carried out in four dimensions, since the optimal kissing configuration is unique in four dimensions \cite{dLLdMK2024}.

We can interpret Leech's construction in terms of filling holes. The central cross section is the $D_4$ root system
\[
(\overbrace{\pm 1, \pm 1, 0, 0}^{\text{permute}},0)
\]
or equivalently the vertices of a regular $24$-cell. The holes in this spherical code form the dual $24$-cell, which is the union of three cross polytopes, namely
\[
(\overbrace{\pm 1, 0, 0, 0}^{\text{permute}},0), \quad (\overbrace{\pm \tfrac{1}{2}, \pm \tfrac{1}{2}, \pm \tfrac{1}{2}, \pm \tfrac{1}{2}}^{\text{odd \# of $-$ signs}},0), \quad \text{and} \quad (\overbrace{\pm \tfrac{1}{2}, \pm \tfrac{1}{2}, \pm \tfrac{1}{2}, \pm \tfrac{1}{2}}^{\text{even \# of $-$ signs}},0).
\]
The other two layers lie directly above or below some of these points, so that they are nestled into the holes in the central cross section. The triality symmetry of the $24$-cell, i.e., the orthogonal transformation
\[
\frac{1}{2}\begin{pmatrix}1 & 1 & 1 & 1\\
1 & 1 & -1 & -1\\
1 & -1 & 1 & -1\\
-1 & 1 & 1 & -1
\end{pmatrix},
\]
cyclically permutes these cross polytopes in four dimensions, so they all play the same role. Up to isometry, there are only two kissing configurations that can be obtained in this way: $D_5$, which has reflection symmetry across the hyperplane, and $L_5$, which does not. One can see that they are distinct by counting antipodal points: $D_5$ is closed under multiplication by $-1$, while $L_5$ is not. Instead, only $24$ of the points in $L_5$ come in antipodal pairs, namely those in $D_4$.

In 2023, Sz\"oll\H{o}si \cite{szollosi2023} found a new kissing configuration $Q_5$ via a computer search. It can be described similarly to $L_5$, but using different cross sections: the $D_5$ root system contains $20$ vectors with a coordinate sum of $0$ and $10$ each with coordinate sums of $\pm 2$. Here the central cross section is the $A_4$ root system. This cross section is suboptimal as a four-dimensional kissing configuration, but it nevertheless extends to a seemingly optimal configuration in five dimensions. The layer with coordinate sum~$2$ consists of the antipodes of the points with coordinate sum~$-2$, but it is not the reflection of that layer across the hyperplane with coordinate sum~$0$. Reflection across this hyperplane consists of the map
\begin{equation}\label{eqreflect}
    v \mapsto v - 2 \langle v,w\rangle w,
\end{equation}
where $w = (1,1,1,1,1)/\sqrt{5}$ and $\langle \cdot,\cdot\rangle$ denotes the usual inner product. Equivalently, when applied to a vector with coordinate sum $s$, the reflection subtracts $2s/5$ from each coordinate. The $Q_5$ kissing configuration replaces the layer that has coordinate sum~$2$ with the reflection of the layer with coordinate sum~$-2$. Only $20$ points in $Q_5$ come in antipodal pairs, namely those in $A_4$, and so it cannot be isometric to $D_5$ or $L_5$.

Our new configuration, which we call $R_5$, can be obtained as follows by modifying $L_5$. The $L_5$ configuration behaves similarly to $D_5$ with respect to coordinate sums: there are $20$ points with coordinate sum $0$ and $10$ each with $\pm 2$. The central cross section is no longer the $A_4$ root system, but rather a modification that replaces four adjacent points with a different regular tetrahedron. Again we can replace one of the adjacent layers with the other's reflection, to obtain $R_5$ from $L_5$. Only $12$ of the points in $R_5$ come in antipodal pairs, and so it cannot be isometric to any of the three other configurations. This yields the following theorem:

\begin{theorem}
There are at least four non-isometric kissing configurations of $40$ points in five dimensions, namely $D_5$, $L_5$, $Q_5$, and $R_5$.
\end{theorem}

\begin{table}
\caption{The inner products that occur in each of the five-dimensional kissing configurations we study. The column labeled $t$ lists the number of unordered pairs of distinct points with inner product $t$, if the points are all normalized to be unit vectors.}
\label{5dinnerproducts}
\centering
\begin{tabular}{cccccccccc}
\toprule
Configuration & $-1$ & $-4/5$ & $-3/4$ & $-1/2$ & $-3/10$ & $-1/4$ & $0$ & $1/5$ & $1/2$\\
\midrule
$D_5$ & $20$ & $0$ & $0$ & $240$ & $0$ & $0$ & $280$ & $0$ & $240$\\
$L_5$ & $12$ & $0$ & $32$ & $192$ & $0$ & $32$ & $272$ & $0$ & $240$\\
$Q_5$ & $10$ & $30$ & $0$ & $180$ & $60$ & $0$ & $250$ & $10$ & $240$\\
$R_5$ & $6$ & $30$ & $20$ & $144$ & $60$ & $28$ & $242$ & $10$ & $240$\\
\bottomrule
\end{tabular}
\end{table}

\begin{table}
\caption{The four five-dimensional kissing configurations with $40$ points. The subdivisions in the list are $D_5$, $L_5$, $Q_5$, and $R_5$, respectively, and the points with fractional coordinates are the only ones that differ from those in $D_5$.}
\label{5dkissing}
\centering
\begin{tabular}{cccc}
\toprule
$\scriptstyle (1,0,0,0,1)$ & $\scriptstyle (-1,0,0,0,1)$ & $\scriptstyle (1,0,0,0,-1)$ & $\scriptstyle (-1,0,0,0,-1)$ \\
$\scriptstyle (0,1,0,0,1)$ & $\scriptstyle (0,-1,0,0,1)$ & $\scriptstyle (0,1,0,0,-1)$ & $\scriptstyle (0,-1,0,0,-1)$ \\
$\scriptstyle (0,0,1,0,1)$ & $\scriptstyle (0,0,-1,0,1)$ & $\scriptstyle (0,0,1,0,-1)$ & $\scriptstyle (0,0,-1,0,-1)$ \\
$\scriptstyle (0,0,0,1,1)$ & $\scriptstyle (0,0,0,-1,1)$ & $\scriptstyle (0,0,0,1,-1)$ & $\scriptstyle (0,0,0,-1,-1)$ \\
$\scriptstyle {(1,1,0,0,0)}$ & $\scriptstyle (-1,1,0,0,0)$ & $\scriptstyle (1,-1,0,0,0)$ & $\scriptstyle (-1,-1,0,0,0)$ \\
$\scriptstyle {(1,0,1,0,0)}$ & $\scriptstyle (-1,0,1,0,0)$ & $\scriptstyle (1,0,-1,0,0)$ & $\scriptstyle (-1,0,-1,0,0)$ \\
$\scriptstyle (1,0,0,1,0)$ & $\scriptstyle (-1,0,0,1,0)$ & $\scriptstyle (1,0,0,-1,0)$ & $\scriptstyle (-1,0,0,-1,0)$ \\
$\scriptstyle (0,1,1,0,0)$ & $\scriptstyle (0,-1,1,0,0)$ & $\scriptstyle (0,1,-1,0,0)$ & $\scriptstyle (0,-1,-1,0,0)$ \\
$\scriptstyle (0,1,0,1,0)$ & $\scriptstyle (0,-1,0,1,0)$ & $\scriptstyle (0,1,0,-1,0)$ & $\scriptstyle (0,-1,0,-1,0)$ \\
$\scriptstyle (0,0,1,1,0)$ & $\scriptstyle (0,0,-1,1,0)$ & $\scriptstyle (0,0,1,-1,0)$ & $\scriptstyle (0,0,-1,-1,0)$ \\
\midrule
$\scriptstyle {(-0.5,0.5,0.5,0.5,1)}$ & $\scriptstyle {(0.5,-0.5,-0.5,-0.5,1)}$ & $\scriptstyle (1,0,0,0,-1)$ & $\scriptstyle (-1,0,0,0,-1)$ \\
$\scriptstyle {(0.5,-0.5,0.5,0.5,1)}$ & $\scriptstyle {(-0.5,0.5,-0.5,-0.5,1)}$ & $\scriptstyle (0,1,0,0,-1)$ & $\scriptstyle (0,-1,0,0,-1)$ \\
$\scriptstyle {(0.5,0.5,-0.5,0.5,1)}$ & $\scriptstyle {(-0.5,-0.5,0.5,-0.5,1)}$ & $\scriptstyle (0,0,1,0,-1)$ & $\scriptstyle (0,0,-1,0,-1)$ \\
$\scriptstyle {(0.5,0.5,0.5,-0.5,1)}$ & $\scriptstyle {(-0.5,-0.5,-0.5,0.5,1)}$ & $\scriptstyle (0,0,0,1,-1)$ & $\scriptstyle (0,0,0,-1,-1)$ \\
$\scriptstyle {(1,1,0,0,0)}$ & $\scriptstyle (-1,1,0,0,0)$ & $\scriptstyle (1,-1,0,0,0)$ & $\scriptstyle (-1,-1,0,0,0)$ \\
$\scriptstyle (1,0,1,0,0)$ & $\scriptstyle (-1,0,1,0,0)$ & $\scriptstyle (1,0,-1,0,0)$ & $\scriptstyle (-1,0,-1,0,0)$ \\
$\scriptstyle (1,0,0,1,0)$ & $\scriptstyle (-1,0,0,1,0)$ & $\scriptstyle (1,0,0,-1,0)$ & $\scriptstyle (-1,0,0,-1,0)$ \\
$\scriptstyle (0,1,1,0,0)$ & $\scriptstyle (0,-1,1,0,0)$ & $\scriptstyle (0,1,-1,0,0)$ & $\scriptstyle (0,-1,-1,0,0)$ \\
$\scriptstyle (0,1,0,1,0)$ & $\scriptstyle (0,-1,0,1,0)$ & $\scriptstyle (0,1,0,-1,0)$ & $\scriptstyle (0,-1,0,-1,0)$ \\
$\scriptstyle (0,0,1,1,0)$ & $\scriptstyle (0,0,-1,1,0)$ & $\scriptstyle (0,0,1,-1,0)$ & $\scriptstyle (0,0,-1,-1,0)$ \\
\midrule
$\scriptstyle {{(-0.2,0.8,0.8,0.8,-0.2)}}$ & $\scriptstyle (-1,0,0,0,1)$ & $\scriptstyle (1,0,0,0,-1)$ & $\scriptstyle (-1,0,0,0,-1)$ \\
$\scriptstyle {{(0.8,-0.2,0.8,0.8,-0.2)}}$ & $\scriptstyle (0,-1,0,0,1)$ & $\scriptstyle (0,1,0,0,-1)$ & $\scriptstyle (0,-1,0,0,-1)$ \\
$\scriptstyle {(0.8,0.8,-0.2,0.8,-0.2)}$ & $\scriptstyle (0,0,-1,0,1)$ & $\scriptstyle (0,0,1,0,-1)$ & $\scriptstyle (0,0,-1,0,-1)$ \\
$\scriptstyle {(0.8,0.8,0.8,-0.2,-0.2)}$ & $\scriptstyle (0,0,0,-1,1)$ & $\scriptstyle (0,0,0,1,-1)$ & $\scriptstyle (0,0,0,-1,-1)$ \\
$\scriptstyle {(-0.2,-0.2,0.8,0.8,0.8)}$ & $\scriptstyle {(-1,1,0,0,0)}$ & $\scriptstyle (1,-1,0,0,0)$ & $\scriptstyle (-1,-1,0,0,0)$ \\
$\scriptstyle {(-0.2,0.8,-0.2,0.8,0.8)}$ & $\scriptstyle (-1,0,1,0,0)$ & $\scriptstyle (1,0,-1,0,0)$ & $\scriptstyle (-1,0,-1,0,0)$ \\
$\scriptstyle {(-0.2,0.8,0.8,-0.2,0.8)}$ & $\scriptstyle (-1,0,0,1,0)$ & $\scriptstyle (1,0,0,-1,0)$ & $\scriptstyle (-1,0,0,-1,0)$ \\
$\scriptstyle {(0.8,-0.2,-0.2,0.8,0.8)}$ & $\scriptstyle (0,-1,1,0,0)$ & $\scriptstyle (0,1,-1,0,0)$ & $\scriptstyle (0,-1,-1,0,0)$ \\
$\scriptstyle {(0.8,-0.2,0.8,-0.2,0.8)}$ & $\scriptstyle (0,-1,0,1,0)$ & $\scriptstyle (0,1,0,-1,0)$ & $\scriptstyle (0,-1,0,-1,0)$ \\
$\scriptstyle {(0.8,0.8,-0.2,-0.2,0.8)}$ & $\scriptstyle (0,0,-1,1,0)$ & $\scriptstyle (0,0,1,-1,0)$ & $\scriptstyle (0,0,-1,-1,0)$ \\ 
\midrule
$\scriptstyle {(-0.2,0.8,0.8,0.8,-0.2)}$ & $\scriptstyle {{(0.5,-0.5,-0.5,-0.5,1)}}$ & $\scriptstyle (1,0,0,0,-1)$ & $\scriptstyle (-1,0,0,0,-1)$ \\
$\scriptstyle {(0.8,-0.2,0.8,0.8,-0.2)}$ & $\scriptstyle {(-0.5,0.5,-0.5,-0.5,1)}$ & $\scriptstyle (0,1,0,0,-1)$ & $\scriptstyle (0,-1,0,0,-1)$ \\
$\scriptstyle {(0.8,0.8,-0.2,0.8,-0.2)}$ & $\scriptstyle {(-0.5,-0.5,0.5,-0.5,1)}$ & $\scriptstyle (0,0,1,0,-1)$ & $\scriptstyle (0,0,-1,0,-1)$ \\
$\scriptstyle {(0.8,0.8,0.8,-0.2,-0.2)}$ & $\scriptstyle {(-0.5,-0.5,-0.5,0.5,1)} $ & $\scriptstyle (0,0,0,1,-1)$ & $\scriptstyle (0,0,0,-1,-1)$ \\
$\scriptstyle {(-0.2,-0.2,0.8,0.8,0.8)}$ & $\scriptstyle (-1,1,0,0,0)$ & $\scriptstyle (1,-1,0,0,0)$ & $\scriptstyle (-1,-1,0,0,0)$ \\
$\scriptstyle {(-0.2,0.8,-0.2,0.8,0.8)}$ & $\scriptstyle (-1,0,1,0,0)$ & $\scriptstyle (1,0,-1,0,0)$ & $\scriptstyle (-1,0,-1,0,0)$ \\
$\scriptstyle {(-0.2,0.8,0.8,-0.2,0.8)}$ & $\scriptstyle (-1,0,0,1,0)$ & $\scriptstyle (1,0,0,-1,0)$ & $\scriptstyle (-1,0,0,-1,0)$ \\
$\scriptstyle {(0.8,-0.2,-0.2,0.8,0.8)}$ & $\scriptstyle (0,-1,1,0,0)$ & $\scriptstyle (0,1,-1,0,0)$ & $\scriptstyle (0,-1,-1,0,0)$ \\
$\scriptstyle {(0.8,-0.2,0.8,-0.2,0.8)}$ & $\scriptstyle (0,-1,0,1,0)$ & $\scriptstyle (0,1,0,-1,0)$ & $\scriptstyle (0,-1,0,-1,0)$ \\
$\scriptstyle {{(0.8,0.8,-0.2,-0.2,0.8)}}$ & $\scriptstyle (0,0,-1,1,0)$ & $\scriptstyle (0,0,1,-1,0)$ & $\scriptstyle (0,0,-1,-1,0)$ \\
\bottomrule
\end{tabular}
\end{table}

Table~\ref{5dinnerproducts} lists the inner products that occur in these configurations and their multiplicities, while Table~\ref{5dkissing} lists coordinates for the points themselves. The symmetry groups of $D_5$, $L_5$, $Q_5$, and $R_5$ have orders $3840$, $384$, $240$, and $48$, respectively. Specifically, the symmetries of $D_5$ are the signed permutations of the coordinates (equivalently, the $1920$ elements of the $D_5$ Weyl group times two Dynkin diagram automorphisms), and those of $L_5$ are the signed permutations that fix the fifth coordinate. The symmetries of $Q_5$ are generated by the permutations of the coordinates and reflection across the hyperplane, and the symmetries of $R_5$ are generated by the permutations of the coordinates that fix the fifth coordinate and reflection across the hyperplane.

\section{Six-dimensional kissing configurations}
\label{sec:6dkissing}

The $D_5$ and $L_5$ kissing configurations extend to seemingly optimal six-dimensional kissing configurations of $72$ points, but we have been unable to extend $Q_5$ and $R_5$. To see what goes wrong, we begin by reviewing what happens with $D_5$ and $L_5$. 

There are $32$ deep holes in the $D_5$ kissing configuration, namely the points
\[
(\pm \sqrt{2/5}, \dots, \pm \sqrt{2/5})
\]
if we use vectors with squared norm $2$ as above. These holes have inner product at most $2\sqrt{2/5}$ with the points in the $D_5$ kissing configuration, and no other points can be so far away from $D_5$. We can partition them into two subsets $S_\text{even}$ and $S_\text{odd}$ of size $16$, namely those with an even or odd number of minus signs, so that the maximal inner product within each of these subsets is $2/5$. One can check that the $E_6$ kissing configuration is the union
\[
(\sqrt{5/8} \,S_{\text{even}} \times \{-\sqrt{3}/2\}) \cup (D_5 \times \{0\}) \cup (\sqrt{5/8}\,S_{\text{odd}} \times \{\sqrt{3}/2\}),
\]
which has $D_5$ as a central cross section and two parallel layers with $16$ points nestled into the deep holes. If instead we form the union
\[
(\sqrt{5/8} \,S_{\text{even}} \times \{-\sqrt{3}/2\}) \cup (D_5 \times \{0\}) \cup (\sqrt{5/8}\,S_{\text{even}} \times \{\sqrt{3}/2\}),
\]
which has reflection symmetry across the central hyperplane, then we obtain a kissing configuration discovered by Leech \cite{leech1969}. (The other two possibilities, odd-even and odd-odd, are isometric to these cases.)

A similar construction works if we start instead with $L_5$. For the construction of $L_5$ from the previous section, the deep holes with negative fifth coordinate are the same as in $D_5$, while those with positive fifth coordinate are
\[
\sqrt{\frac{2}{5}}\,(\overbrace{\pm 2, 0, 0, 0}^{\text{permute}},1) \qquad\text{and}\qquad
\sqrt{\frac{2}{5}}\,(\overbrace{\pm 1, \pm 1, \pm 1, \pm 1}^{\text{even \# of $-$ signs}},1).
\]
Again we can partition these holes into two $16$-element subsets with maximal inner product $2/5$ within each subset, with one subset being the same set $S_{\text{even}}$ from the $D_5$ case and the other being its complement. Now, however, there is also a third $16$-element subset with maximal inner product $2/5$, which combines the points
\[
\sqrt{\frac{2}{5}}\,(\overbrace{\pm 2, 0, 0, 0}^{\text{permute}},1)
\]
with the eight elements of $S_{\text{even}}$ that have a negative fifth coordinate. We can imitate the $D_5$ construction by using these $16$-element sets to create parallel layers. If we choose the two complementary sets, we obtain Leech's six-dimensional kissing configuration again. The other cases yield two distinct six-dimensional kissing configurations, which were discovered by Conway and Sloane \cite{cs1995}. Specifically, one of these configurations arises only if we use the third $16$-element subset on both sides of $L_5$, and all the remaining cases give the other configuration.

If we start with $Q_5$ or $R_5$, it initially seems that the same approach might produce an extension to six dimensions, but it cannot be fully carried out. These configurations again have $32$ deep holes: those on one side of the hyperplane with coordinate sum~$0$ are the same as in $D_5$ or $L_5$, respectively, while those on the other side are the reflections across this hyperplane. However, continuing the construction requires a set of $16$ holes with maximal inner product $2/5$ between them, and one can check by depth-first search that no such set exists. We have not ruled out the possibility that some other construction might give a six-dimensional kissing configuration with $Q_5$ or $R_5$ as a cross section, but it seems unlikely:

\begin{conjecture}
No six-dimensional kissing configuration with $Q_5$ or $R_5$ as a cross section can contain $72$ or more points.
\end{conjecture}

The existence of $Q_5$ and $R_5$ suggests that analogous kissing configurations may also exist in six or seven dimensions,\footnote{The eight-dimensional kissing problem has a unique solution \cite{bannaisloane1981}.} but we have had no success in finding such configurations. While there are various ways to divide six- or seven-dimensional configurations into layers, we have not found any sufficiently close analogues of the $A_4$ cross section of $D_5$ that leads to $Q_5$. Perhaps this cross section is a special phenomenon in five dimensions, or perhaps some new insight is needed for a generalization.

\section{From kissing configurations to sphere packings}
\label{sec:5dpacking}

To extend these kissing configurations to Euclidean sphere packings, we will imitate Conway and Sloane's construction of tight packings \cite{cs1995}. Their five-dimensional packings were all formed by stacking four-dimensional layers, specifically translates of the $D_4$ root lattice. Instead, we will stack three-dimensional layers given by translates of the $D_3$ root lattice.

\subsection{Packings that fiber over $D_3$}\label{section4.1}

Recall that $D_3$ consists of the points $(x,y,z)$ for which $x$, $y$, and $z$ are integers with $x+y+z$ even. Because $D_3$ is an integral lattice, it is a sublattice of its dual lattice $D_3^*$, which consists of the translates of $D_3$ by the vectors
\begin{align*}
t_0 &= (0,0,0),\\
t_1 &= (\tfrac{1}{2},\tfrac{1}{2},\tfrac{1}{2}),\\
t_2 &= (0,0,1), \text{ and}\\
t_3 &= (\tfrac{1}{2},\tfrac{1}{2},-\tfrac{1}{2}).
\end{align*}
In other words, $D_3+t_0$ is the $D_3$ root lattice itself, $D_3+t_2$ consists of the deep holes in the lattice (the points furthest from the lattice, with distance $1$), and the remaining two translates consist of the shallow holes (local but not global maxima for distance from the lattice, with distance $\sqrt{3}/2$).

To form a five-dimensional packing, we will use a \emph{four-colored point configuration} in $\R^2$, by which we mean an infinite discrete point set in $\R^2$ in which each point is labeled with the color $0$, $1$, $2$, or $3$. We call two colors \emph{adjacent} if they differ by $1$ modulo $4$, and \emph{opposite} if they differ by $2$. If the point $v$ is colored $i$, then we center spheres of radius $\sqrt{2}/2$ at the points in $\{v\} \times (D_3 + t_i)$. In order for the spheres not to overlap, the four-colored point configuration must satisfy the following constraints:
\begin{enumerate}
\item Two distinct points with the same color must have distance at least $\sqrt{2}$ between them.

\item Two points with adjacent colors must have distance at least $\sqrt{5}/2$ between them.

\item Two points with opposite colors must have distance at least $1$ between them.
\end{enumerate}
A \emph{valid} coloring is one that satisfies these constraints. If a valid four-colored point configuration has $\delta$ points per unit area in $\R^2$, then the corresponding sphere packing in $\R^5$ has packing density given by 
\[
\frac{\pi^{5/2}}{(5/2)!} \left(\frac{\sqrt{2}}{2}\right)^5 \cdot \delta \cdot \frac{1}{2} = \frac{\pi^2 \sqrt{2}}{30} \, \delta.
\]
Here, we are multiplying the volume of a sphere of radius $\sqrt{2}/2$ in $\R^5$ with the point densities of the four-colored point configuration in $\R^2$ and the $D_3$ root lattice in $\R^3$. In particular, matching the packing density of the $D_5$ root lattice amounts to achieving $\delta=1$. Presumably it is impossible to achieve $\delta>1$:

\begin{conjecture}
For every valid four-colored point configuration $C$ in $\R^2$,
\[
\limsup_{r \to \infty} \frac{\#\{x \in C : |x| \le r\}}{\pi r^2} \le 1.
\]
\end{conjecture}

We do not have a proof of this conjecture, but because it is fundamentally two-dimensional, we expect it to be more tractable than the full optimality of $D_5$ as a sphere packing.

To achieve $\delta=1$, we can tile the plane with two types of triangles: isosceles triangles with side lengths $\sqrt{5}/2,\sqrt{5}/2,1$ or $1,1,\sqrt{2}$, which we call $\Delta_1$ and $\Delta_2$. Both of these triangles have area $1/2$, and it follows that in every edge-to-edge tiling, there is one vertex per unit area.\footnote{The angles in a triangle add up to $\pi$, which means they account for half of the total angle surrounding a single vertex. Therefore there must be half a vertex per triangle on average in any edge-to-edge tiling.} Our five-dimensional sphere packings are obtained from valid four-colorings of the vertices of these tilings. Specifically, vertices at distance~$\sqrt{2}$ have the same color, those at distance~$\sqrt{5}/2$ have adjacent colors, and those at distance~$1$ have opposite colors.

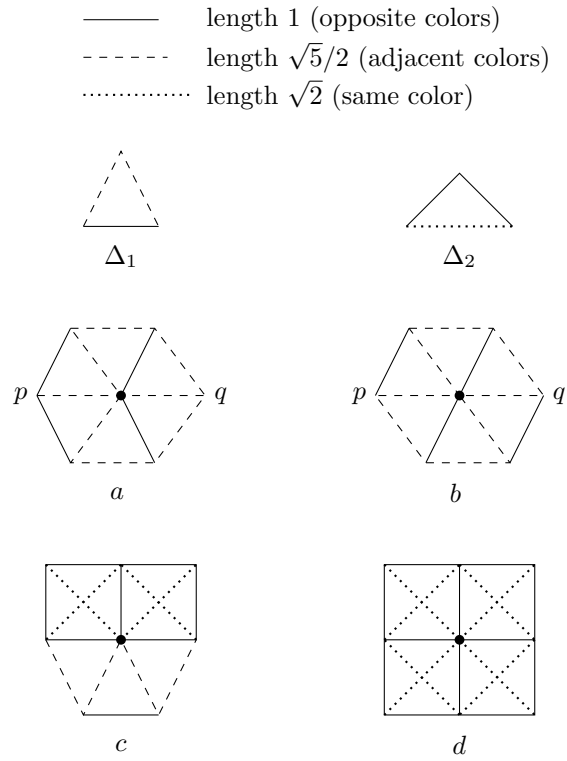
\begin{figure}
\centering
\begin{tikzpicture}
\begin{scope}[shift={(0,1.75)}]
\draw (0,1)--(1,1);
\draw ({sqrt(2)+0.1},1) node[right] {length $1$ (opposite colors)};
\draw[dashed] (0,0.5)--({sqrt(5)/2},0.5);
\draw ({sqrt(2)+0.1},0.5) node[right] {length $\sqrt{5}/2$ (adjacent colors)};
\draw[dotted,thick] (0,0)--({sqrt(2)},0);
\draw ({sqrt(2)+0.1},0) node[right] {length $\sqrt{2}$ (same color)};
\end{scope}
\draw (0,0)--(1,0);
\draw[dashed] (0,0)--(0.5,1)--(1,0);
\draw (0.5,-0.4) node {$\Delta_1$};
\begin{scope}[shift={(5,0)}]
\draw[dotted,thick] ({-sqrt(2)/2},0)--({sqrt(2)/2},0);
\draw ({-sqrt(2)/2},0)--(0,{sqrt(2)/2})--({sqrt(2)/2},0);
\draw (0,-0.4) node {$\Delta_2$};
\end{scope}
\begin{scope}[shift={(0.5,-2.25)}]
\fill (0,0) circle (0.065);
\draw (0,0)--({1/sqrt(5)},{2/sqrt(5)});
\draw (0,0)--({1/sqrt(5)},{-2/sqrt(5)});
\draw ({-sqrt(5)/2},0)--({-sqrt(5)/2+1/sqrt(5)},{2/sqrt(5)});
\draw ({-sqrt(5)/2},0)--({-sqrt(5)/2+1/sqrt(5)},{-2/sqrt(5)});
\draw[dashed] ({-sqrt(5)/2+1/sqrt(5)},{2/sqrt(5)})--({1/sqrt(5)},{2/sqrt(5)})--({sqrt(5)/2},0)--({1/sqrt(5)},{-2/sqrt(5)})--({-sqrt(5)/2+1/sqrt(5)},{-2/sqrt(5)})--(0,0)--cycle;
\draw[dashed] ({-sqrt(5)/2},0)--({sqrt(5)/2},0);
\draw ({sqrt(5)/2},0) node[right] {$q$};
\draw ({-sqrt(5)/2},0) node[left] {$p$};
\draw ({1/sqrt(5)-0.5},{-2/sqrt(5)-0.4}) node {$a$};
\end{scope}
\begin{scope}[shift={({5},-2.25)}]
\fill (0,0) circle (0.065);
\draw (0,0)--({1/sqrt(5)},{2/sqrt(5)});
\draw (0,0)--({-1/sqrt(5)},{-2/sqrt(5)});
\draw ({-sqrt(5)/2},0)--({-sqrt(5)/2+1/sqrt(5)},{2/sqrt(5)});
\draw ({sqrt(5)/2},0)--({sqrt(5)/2-1/sqrt(5)},{-2/sqrt(5)});
\draw[dashed] ({-sqrt(5)/2+1/sqrt(5)},{2/sqrt(5)})--({1/sqrt(5)},{2/sqrt(5)})--({sqrt(5)/2},0)--(0,0)--({sqrt(5)/2-1/sqrt(5)},{-2/sqrt(5)})--({-1/sqrt(5)},{-2/sqrt(5)})--({-sqrt(5)/2},0)--(0,0)--cycle;
\draw ({sqrt(5)/2},0) node[right] {$q$};
\draw ({-sqrt(5)/2},0) node[left] {$p$};
\draw ({1/sqrt(5)-0.5},{-2/sqrt(5)-0.4}) node {$b$};
\end{scope}
\begin{scope}[shift={(-0.5,-6.5)}]
\fill (1,1) circle (0.065);
\draw (0,1)--(0,2)--(2,2)--(2,1)--cycle;
\draw (1,1)--(1,2);
\draw[dotted,thick] (1,1)--(2,2);
\draw[dotted,thick] (0,2)--(1,1);
\draw[dotted,thick] (2,1)--(1,2);
\draw[dotted,thick] (1,2)--(0,1);
\draw (0.5,0)--(1.5,0);
\draw[dashed] (0,1)--(0.5,0)--(1,1)--(1.5,0)--(2,1);
\draw ({1},{-0.4}) node {$c$};
\end{scope}
\begin{scope}[shift={({4},-6.5)}]
\fill (1,1) circle (0.065);
\draw (0,0)--(2,0)--(2,2)--(0,2)--cycle;
\draw (1,0)--(1,2);
\draw (0,1)--(2,1);
\draw[dotted,thick] (0,0)--(1,1);
\draw[dotted,thick] (1,1)--(2,2);
\draw[dotted,thick] (0,2)--(1,1);
\draw[dotted,thick] (1,1)--(2,0);
\draw[dotted,thick] (0,1)--(1,2);
\draw[dotted,thick] (1,2)--(2,1);
\draw[dotted,thick] (2,1)--(1,0);
\draw[dotted,thick] (1,0)--(0,1);
\draw ({1},{-0.4}) node {$d$};
\end{scope}
\end{tikzpicture}
\caption{The four possible ways to arrange triangles $\Delta_1$ and $\Delta_2$ around a point, labeled $a$ through $d$. In arrangements $a$ and $b$, two of the neighboring points are labeled $p$ and $q$.}
\label{figurearrangements}
\end{figure}

\begin{table}
\caption{The kissing configurations at the points marked with dots in Figure~\ref{figurearrangements}.}
\label{tablekissing}
\centering
\begin{tabular}{ccc}
\toprule
Local arrangement & Condition & Kissing configuration\\
\midrule
$a$ & $p$ and $q$ same color & $R_5$\\
$a$ & $p$ and $q$ different colors & $Q_5$\\
\midrule
$b$ & $p$ and $q$ same color & $L_5$\\
$b$ & $p$ and $q$ different colors  & $D_5$\\
\midrule
$c$ & & $L_5$\\
\midrule
$d$ & & $D_5$\\
\bottomrule
\end{tabular}
\end{table}

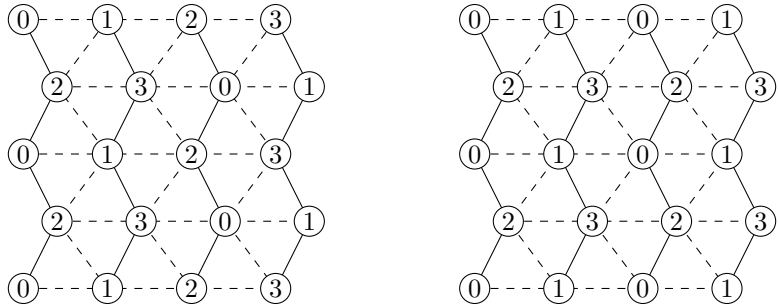
\begin{figure}
\centering
\begin{tikzpicture}
\draw[dashed] (0,0)--({3*sqrt(5)/2},0);
\draw[dashed] ({1/sqrt(5)},{2/sqrt(5)})--({1/sqrt(5)+3*sqrt(5)/2},{2/sqrt(5)});
\draw[dashed] (0,{4/sqrt(5)})--({3*sqrt(5)/2},{4/sqrt(5)});
\draw[dashed] ({1/sqrt(5)},{6/sqrt(5)})--({1/sqrt(5)+3*sqrt(5)/2},{6/sqrt(5)});
\draw[dashed] (0,{8/sqrt(5)})--({3*sqrt(5)/2},{8/sqrt(5)});
\draw[dashed] ({sqrt(5)/2},{8/sqrt(5)})--({1/sqrt(5)},{6/sqrt(5)})--({sqrt(5)/2},{4/sqrt(5)})--({1/sqrt(5)},{2/sqrt(5)})--({sqrt(5)/2},0);
\draw[dashed] ({2*sqrt(5)/2},{8/sqrt(5)})--({sqrt(5)/2+1/sqrt(5)},{6/sqrt(5)})--({2*sqrt(5)/2},{4/sqrt(5)})--({sqrt(5)/2+1/sqrt(5)},{2/sqrt(5)})--({2*sqrt(5)/2},0);
\draw[dashed] ({3*sqrt(5)/2},{8/sqrt(5)})--({2*sqrt(5)/2+1/sqrt(5)},{6/sqrt(5)})--({3*sqrt(5)/2},{4/sqrt(5)})--({2*sqrt(5)/2+1/sqrt(5)},{2/sqrt(5)})--({3*sqrt(5)/2},0);
\draw (0,0)--({1/sqrt(5)},{2/sqrt(5)})--(0,{4/sqrt(5)})--({1/sqrt(5)},{6/sqrt(5)})--(0,{8/sqrt(5)});
\draw ({sqrt(5)/2},0)--({sqrt(5)/2+1/sqrt(5)},{2/sqrt(5)})--({sqrt(5)/2},{4/sqrt(5)})--({sqrt(5)/2+1/sqrt(5)},{6/sqrt(5)})--({sqrt(5)/2},{8/sqrt(5)});
\draw ({sqrt(5)},0)--({sqrt(5)+1/sqrt(5)},{2/sqrt(5)})--({sqrt(5)},{4/sqrt(5)})--({sqrt(5)+1/sqrt(5)},{6/sqrt(5)})--({sqrt(5)},{8/sqrt(5)});
\draw ({3*sqrt(5)/2},0)--({3*sqrt(5)/2+1/sqrt(5)},{2/sqrt(5)})--({3*sqrt(5)/2},{4/sqrt(5)})--({3*sqrt(5)/2+1/sqrt(5)},{6/sqrt(5)})--({3*sqrt(5)/2},{8/sqrt(5)});
\fill[white] (0,0) circle (0.2); \draw (0,0) circle (0.2); \draw (0,0) node {$0$};
\fill[white] ({sqrt(5)/2},0) circle (0.2); \draw ({sqrt(5)/2},0) circle (0.2); \draw ({sqrt(5)/2},0) node {$1$};
\fill[white] ({2*sqrt(5)/2},0) circle (0.2); \draw ({2*sqrt(5)/2},0) circle (0.2); \draw ({2*sqrt(5)/2},0) node {$2$};
\fill[white] ({3*sqrt(5)/2},0) circle (0.2); \draw ({3*sqrt(5)/2},0) circle (0.2); \draw ({3*sqrt(5)/2},0) node {$3$};
\fill[white] ({1/sqrt(5)},{2/sqrt(5)}) circle (0.2); \draw ({1/sqrt(5)},{2/sqrt(5)}) circle (0.2); \draw ({1/sqrt(5)},{2/sqrt(5)}) node {$2$};
\fill[white] ({1/sqrt(5)+sqrt(5)/2},{2/sqrt(5)}) circle (0.2); \draw ({1/sqrt(5)+sqrt(5)/2},{2/sqrt(5)}) circle (0.2); \draw ({1/sqrt(5)+sqrt(5)/2},{2/sqrt(5)}) node {$3$};
\fill[white] ({1/sqrt(5)+2*sqrt(5)/2},{2/sqrt(5)}) circle (0.2); \draw ({1/sqrt(5)+2*sqrt(5)/2},{2/sqrt(5)}) circle (0.2); \draw ({1/sqrt(5)+2*sqrt(5)/2},{2/sqrt(5)}) node {$0$};
\fill[white] ({1/sqrt(5)+3*sqrt(5)/2},{2/sqrt(5)}) circle (0.2); \draw ({1/sqrt(5)+3*sqrt(5)/2},{2/sqrt(5)}) circle (0.2); \draw ({1/sqrt(5)+3*sqrt(5)/2},{2/sqrt(5)}) node {$1$};
\fill[white] (0,{4/sqrt(5)}) circle (0.2); \draw (0,{4/sqrt(5)}) circle (0.2); \draw (0,{4/sqrt(5)}) node {$0$};
\fill[white] ({sqrt(5)/2},{4/sqrt(5)}) circle (0.2); \draw ({sqrt(5)/2},{4/sqrt(5)}) circle (0.2); \draw ({sqrt(5)/2},{4/sqrt(5)}) node {$1$};
\fill[white] ({2*sqrt(5)/2},{4/sqrt(5)}) circle (0.2); \draw ({2*sqrt(5)/2},{4/sqrt(5)}) circle (0.2); \draw ({2*sqrt(5)/2},{4/sqrt(5)}) node {$2$};
\fill[white] ({3*sqrt(5)/2},{4/sqrt(5)}) circle (0.2); \draw ({3*sqrt(5)/2},{4/sqrt(5)}) circle (0.2); \draw ({3*sqrt(5)/2},{4/sqrt(5)}) node {$3$};
\fill[white] ({1/sqrt(5)},{6/sqrt(5)}) circle (0.2); \draw ({1/sqrt(5)},{6/sqrt(5)}) circle (0.2); \draw ({1/sqrt(5)},{6/sqrt(5)}) node {$2$};
\fill[white] ({1/sqrt(5)+sqrt(5)/2},{6/sqrt(5)}) circle (0.2); \draw ({1/sqrt(5)+sqrt(5)/2},{6/sqrt(5)}) circle (0.2); \draw ({1/sqrt(5)+sqrt(5)/2},{6/sqrt(5)}) node {$3$};
\fill[white] ({1/sqrt(5)+2*sqrt(5)/2},{6/sqrt(5)}) circle (0.2); \draw ({1/sqrt(5)+2*sqrt(5)/2},{6/sqrt(5)}) circle (0.2); \draw ({1/sqrt(5)+2*sqrt(5)/2},{6/sqrt(5)}) node {$0$};
\fill[white] ({1/sqrt(5)+3*sqrt(5)/2},{6/sqrt(5)}) circle (0.2); \draw ({1/sqrt(5)+3*sqrt(5)/2},{6/sqrt(5)}) circle (0.2); \draw ({1/sqrt(5)+3*sqrt(5)/2},{6/sqrt(5)}) node {$1$};
\fill[white] (0,{8/sqrt(5)}) circle (0.2); \draw (0,{8/sqrt(5)}) circle (0.2); \draw (0,{8/sqrt(5)}) node {$0$};
\fill[white] ({sqrt(5)/2},{8/sqrt(5)}) circle (0.2); \draw ({sqrt(5)/2},{8/sqrt(5)}) circle (0.2); \draw ({sqrt(5)/2},{8/sqrt(5)}) node {$1$};
\fill[white] ({2*sqrt(5)/2},{8/sqrt(5)}) circle (0.2); \draw ({2*sqrt(5)/2},{8/sqrt(5)}) circle (0.2); \draw ({2*sqrt(5)/2},{8/sqrt(5)}) node {$2$};
\fill[white] ({3*sqrt(5)/2},{8/sqrt(5)}) circle (0.2); \draw ({3*sqrt(5)/2},{8/sqrt(5)}) circle (0.2); \draw ({3*sqrt(5)/2},{8/sqrt(5)}) node {$3$};
\begin{scope}[shift={(6,0)}]
\draw[dashed] (0,0)--({3*sqrt(5)/2},0);
\draw[dashed] ({1/sqrt(5)},{2/sqrt(5)})--({1/sqrt(5)+3*sqrt(5)/2},{2/sqrt(5)});
\draw[dashed] (0,{4/sqrt(5)})--({3*sqrt(5)/2},{4/sqrt(5)});
\draw[dashed] ({1/sqrt(5)},{6/sqrt(5)})--({1/sqrt(5)+3*sqrt(5)/2},{6/sqrt(5)});
\draw[dashed] (0,{8/sqrt(5)})--({3*sqrt(5)/2},{8/sqrt(5)});
\draw[dashed] ({sqrt(5)/2},{8/sqrt(5)})--({1/sqrt(5)},{6/sqrt(5)})--({sqrt(5)/2},{4/sqrt(5)})--({1/sqrt(5)},{2/sqrt(5)})--({sqrt(5)/2},0);
\draw[dashed] ({2*sqrt(5)/2},{8/sqrt(5)})--({sqrt(5)/2+1/sqrt(5)},{6/sqrt(5)})--({2*sqrt(5)/2},{4/sqrt(5)})--({sqrt(5)/2+1/sqrt(5)},{2/sqrt(5)})--({2*sqrt(5)/2},0);
\draw[dashed] ({3*sqrt(5)/2},{8/sqrt(5)})--({2*sqrt(5)/2+1/sqrt(5)},{6/sqrt(5)})--({3*sqrt(5)/2},{4/sqrt(5)})--({2*sqrt(5)/2+1/sqrt(5)},{2/sqrt(5)})--({3*sqrt(5)/2},0);
\draw (0,0)--({1/sqrt(5)},{2/sqrt(5)})--(0,{4/sqrt(5)})--({1/sqrt(5)},{6/sqrt(5)})--(0,{8/sqrt(5)});
\draw ({sqrt(5)/2},0)--({sqrt(5)/2+1/sqrt(5)},{2/sqrt(5)})--({sqrt(5)/2},{4/sqrt(5)})--({sqrt(5)/2+1/sqrt(5)},{6/sqrt(5)})--({sqrt(5)/2},{8/sqrt(5)});
\draw ({sqrt(5)},0)--({sqrt(5)+1/sqrt(5)},{2/sqrt(5)})--({sqrt(5)},{4/sqrt(5)})--({sqrt(5)+1/sqrt(5)},{6/sqrt(5)})--({sqrt(5)},{8/sqrt(5)});
\draw ({3*sqrt(5)/2},0)--({3*sqrt(5)/2+1/sqrt(5)},{2/sqrt(5)})--({3*sqrt(5)/2},{4/sqrt(5)})--({3*sqrt(5)/2+1/sqrt(5)},{6/sqrt(5)})--({3*sqrt(5)/2},{8/sqrt(5)});
\fill[white] (0,0) circle (0.2); \draw (0,0) circle (0.2); \draw (0,0) node {$0$};
\fill[white] ({sqrt(5)/2},0) circle (0.2); \draw ({sqrt(5)/2},0) circle (0.2); \draw ({sqrt(5)/2},0) node {$1$};
\fill[white] ({2*sqrt(5)/2},0) circle (0.2); \draw ({2*sqrt(5)/2},0) circle (0.2); \draw ({2*sqrt(5)/2},0) node {$0$};
\fill[white] ({3*sqrt(5)/2},0) circle (0.2); \draw ({3*sqrt(5)/2},0) circle (0.2); \draw ({3*sqrt(5)/2},0) node {$1$};
\fill[white] ({1/sqrt(5)},{2/sqrt(5)}) circle (0.2); \draw ({1/sqrt(5)},{2/sqrt(5)}) circle (0.2); \draw ({1/sqrt(5)},{2/sqrt(5)}) node {$2$};
\fill[white] ({1/sqrt(5)+sqrt(5)/2},{2/sqrt(5)}) circle (0.2); \draw ({1/sqrt(5)+sqrt(5)/2},{2/sqrt(5)}) circle (0.2); \draw ({1/sqrt(5)+sqrt(5)/2},{2/sqrt(5)}) node {$3$};
\fill[white] ({1/sqrt(5)+2*sqrt(5)/2},{2/sqrt(5)}) circle (0.2); \draw ({1/sqrt(5)+2*sqrt(5)/2},{2/sqrt(5)}) circle (0.2); \draw ({1/sqrt(5)+2*sqrt(5)/2},{2/sqrt(5)}) node {$2$};
\fill[white] ({1/sqrt(5)+3*sqrt(5)/2},{2/sqrt(5)}) circle (0.2); \draw ({1/sqrt(5)+3*sqrt(5)/2},{2/sqrt(5)}) circle (0.2); \draw ({1/sqrt(5)+3*sqrt(5)/2},{2/sqrt(5)}) node {$3$};
\fill[white] (0,{4/sqrt(5)}) circle (0.2); \draw (0,{4/sqrt(5)}) circle (0.2); \draw (0,{4/sqrt(5)}) node {$0$};
\fill[white] ({sqrt(5)/2},{4/sqrt(5)}) circle (0.2); \draw ({sqrt(5)/2},{4/sqrt(5)}) circle (0.2); \draw ({sqrt(5)/2},{4/sqrt(5)}) node {$1$};
\fill[white] ({2*sqrt(5)/2},{4/sqrt(5)}) circle (0.2); \draw ({2*sqrt(5)/2},{4/sqrt(5)}) circle (0.2); \draw ({2*sqrt(5)/2},{4/sqrt(5)}) node {$0$};
\fill[white] ({3*sqrt(5)/2},{4/sqrt(5)}) circle (0.2); \draw ({3*sqrt(5)/2},{4/sqrt(5)}) circle (0.2); \draw ({3*sqrt(5)/2},{4/sqrt(5)}) node {$1$};
\fill[white] ({1/sqrt(5)},{6/sqrt(5)}) circle (0.2); \draw ({1/sqrt(5)},{6/sqrt(5)}) circle (0.2); \draw ({1/sqrt(5)},{6/sqrt(5)}) node {$2$};
\fill[white] ({1/sqrt(5)+sqrt(5)/2},{6/sqrt(5)}) circle (0.2); \draw ({1/sqrt(5)+sqrt(5)/2},{6/sqrt(5)}) circle (0.2); \draw ({1/sqrt(5)+sqrt(5)/2},{6/sqrt(5)}) node {$3$};
\fill[white] ({1/sqrt(5)+2*sqrt(5)/2},{6/sqrt(5)}) circle (0.2); \draw ({1/sqrt(5)+2*sqrt(5)/2},{6/sqrt(5)}) circle (0.2); \draw ({1/sqrt(5)+2*sqrt(5)/2},{6/sqrt(5)}) node {$2$};
\fill[white] ({1/sqrt(5)+3*sqrt(5)/2},{6/sqrt(5)}) circle (0.2); \draw ({1/sqrt(5)+3*sqrt(5)/2},{6/sqrt(5)}) circle (0.2); \draw ({1/sqrt(5)+3*sqrt(5)/2},{6/sqrt(5)}) node {$3$};
\fill[white] (0,{8/sqrt(5)}) circle (0.2); \draw (0,{8/sqrt(5)}) circle (0.2); \draw (0,{8/sqrt(5)}) node {$0$};
\fill[white] ({sqrt(5)/2},{8/sqrt(5)}) circle (0.2); \draw ({sqrt(5)/2},{8/sqrt(5)}) circle (0.2); \draw ({sqrt(5)/2},{8/sqrt(5)}) node {$1$};
\fill[white] ({2*sqrt(5)/2},{8/sqrt(5)}) circle (0.2); \draw ({2*sqrt(5)/2},{8/sqrt(5)}) circle (0.2); \draw ({2*sqrt(5)/2},{8/sqrt(5)}) node {$0$};
\fill[white] ({3*sqrt(5)/2},{8/sqrt(5)}) circle (0.2); \draw ({3*sqrt(5)/2},{8/sqrt(5)}) circle (0.2); \draw ({3*sqrt(5)/2},{8/sqrt(5)}) node {$1$};
\end{scope}
\end{tikzpicture}
\caption{The uniform packings with kissing configuration $Q_5$ (on the left) and $R_5$ (on the right).}
\label{figureuniform}
\end{figure}

Figure~\ref{figurearrangements} shows the triangles $\Delta_1$ and $\Delta_2$, together with the four ways they can be arranged around a vertex, and Table~\ref{tablekissing} lists the resulting kissing configurations. Figure~\ref{figureuniform} shows two examples of valid four-colorings. In these packings, every local arrangement is of type $a$ from Figure~\ref{figurearrangements}, the colorings are constant within each vertical column, and the colorings repeat horizontally with period four for the packing on the left and two for the packing on the right.

The local analysis from Figure~\ref{figurearrangements} leads to a complete classification of edge-to-edge tilings with the tiles $\Delta_1$ and $\Delta_2$. Among these possibilities, those shown in Figure~\ref{figureuniform} are the only ones that have kissing configuration $Q_5$ everywhere or $R_5$ everywhere. However, other combinations of kissing configurations are possible, such as tilings that have layers of types $a$ and $b$, in which case some spheres have kissing configuration $Q_5$ or $R_5$ and some have kissing configuration $D_5$ or $L_5$. One can combine all four kissing configurations in a single packing by varying the coloring condition in Table~\ref{tablekissing} among the different spheres in a packing.

When a solid line segment from Figure~\ref{figurearrangements} extends to a solid straight line of infinite length in both directions, that corresponds to a $D_4$ cross section of the packing. In particular, tilings with layers of types $b$, $c$, and $d$ (with the $b$ layers rotated so all solid lines are parallel) are equivalent to stacking $D_4$ layers as described by Conway and Sloane in Section~5 of \cite{cs1995}. Thus, these tilings can produce packings with kissing configurations of $D_5$ or $L_5$ or some mix of both, including the $D_5$ root lattice and the three uniform packings with $L_5$ kissing configuration. 

As mentioned in the introduction, the existence of more general packings such as those in Figure~\ref{figureuniform} disproves conjectures of Conway and Sloane \cite{cs1995} and Kuperberg \cite{kuperberg2000}. We propose the following conjecture as a salvage for Conjecture~15 in \cite{kuperberg2000}.

\begin{conjecture} \label{conjecture:5d}
Every weakly recurrent, optimally dense sphere packing in $\R^5$ can be obtained from a valid four-coloring of the vertices of an edge-to-edge tiling of $\R^2$ with triangles congruent to $\Delta_1$ and $\Delta_2$.
\end{conjecture}

Of course Sz\"oll\H{o}si's discovery is cause for skepticism about any such classification. We have no great confidence in Conjecture~\ref{conjecture:5d}, but it describes the current state of our knowledge.

\subsection{Uniform packings}

The packings from Figure~\ref{figureuniform} are noteworthy because they are uniform. In other words, their automorphism groups act transitively on the spheres. In his 1967 paper \cite{leech1967}, Leech found three uniform non-lattice packings with the same density as the $D_5$ root lattice, and our construction adds two more. This list is in fact complete:

\begin{theorem}
Up to isometry, there are only six uniform sphere packings in $\R^5$ that have kissing configuration $D_5$, $L_5$, $Q_5$, or $R_5$, namely one with $D_5$, three with $L_5$, and one each with $Q_5$ or $R_5$.
\end{theorem}

\begin{proof}
Suppose $K_0$ is one of these kissing configurations in a uniform packing, and $v$ is a point in $K_0$. By uniformity, the kissing configuration $K_v$ at $v$ must be isometric to $K_0$, and it has some overlap with $K_0$, namely the points in $K_0$ at squared distance~$2$ from $v$. If this overlap has a unique extension to a kissing configuration around $v$ that is isometric to $K_0$, then $K_v$ is uniquely determined by $K_0$. If the extension is unique for all $v \in K_0$, then that is enough to determine the entire uniform packing by induction. For example, one can check that this is the case for $D_5$, based on the observation that the $D_5$ kissing configuration is closed under taking integer linear combinations that lie on the sphere. However, the other three kissing configurations are more subtle. To analyze them, we rely on a computer search.\footnote{Our implementation of this search is available through DSpace@MIT at \url{https://hdl.handle.net/1721.1/157699}.}

For each point $v \in K_0$, we can determine by depth-first search whether the overlap has a unique extension: we enumerate all the isometric embeddings of the overlap into $K_0$ and check whether they are all equivalent modulo symmetries of $K_0$. Each point $v$ that has a uniquely determined kissing configuration provides additional points of overlap with the kissing configurations that are yet to be determined, which increases our ability to determine them. For $Q_5$ and $R_5$, iterating this procedure determines $K_v$ for every $v \in K_0$, and there is therefore at most one uniform packing. We have constructed one in both cases, as shown in Figure~\ref{figureuniform}.

For $L_5$, this procedure cannot prove uniqueness, because it is not true. However, it shows that $K_v$ is uniquely determined whenever $v$ lies in the $D_4$ cross section specified by $v_5=0$. In other words, the uniform packing must be obtained from $D_4$ cross sections. This problem was analyzed by Conway and Sloane in Section~5 of \cite{cs1995}, where they show that Leech's uniform packings are the only ones that satisfy this condition.
\end{proof}

\begin{table}
\caption{The two uniform packings from Figure~\ref{figureuniform}, with the left one expressed as a 2-periodic packing above and the right one as a 4-periodic packing below.}
\label{tableperiodic}
\centering
\begin{tabular}{AA}
\toprule
\multicolumn{2}{c}{Lattice basis} & \multicolumn{2}{c}{Translation vectors}\\
\midrule
\multicolumn{4}{c}{$Q_5$ uniform packing}\\
v_0 &= (1,-1,0,0,0) & x_0 &=(0,0,0,0,0)\\
v_1 &= (0,1,-1,0,0) & x_1 &= (0,0,0,-1,-1)\\
v_2 &= (0,0,1,-1,0)\\
v_3 &= (0,0,0,1,-1)\\
v_4 &= (0.8,0.8,0.8,0.8,0.8)\\
\midrule
\multicolumn{4}{c}{$R_5$ uniform packing}\\
v_0 &= (1,-1,0,0,0) & x_0 &= (0,0,0,0,0)\\
v_1 &= (0,1,-1,0,0) & x_1 &= (0,0,0,-1,-1)\\
v_2 &= (0,0,1,-1,0) & x_2 &= (0,0,-1,-1,0)\\
v_3 &= (-0.5,-0.5,-0.5,-0.5,2) \quad\ 
& x_3 &= (0,0,0,1,-1) \\
v_4 &= (0.8,0.8,0.8,0.8,0.8)\\
\bottomrule
\end{tabular}
\end{table}

Given the set $\mathcal{P}$ of sphere centers in a periodic sphere packing in $\R^n$, let
\[\Lambda = \{y \in \R^n : \mathcal{P} = \mathcal{P} + y\}\]
be its translational symmetry lattice. Then $\mathcal{P}$ is the union of finitely many translates of $\Lambda$. It is called $m$-periodic if it consists of $m$ translates of $\Lambda$, and our choice of $\Lambda$ ensures that it cannot be written as the union of fewer translates of any other lattice (i.e., $m$ is minimal).

The uniform packing from Figure~\ref{figureuniform} with kissing configuration $Q_5$ is $2$-periodic, and the one with kissing configuration $R_5$ is $4$-periodic, with lattice bases and translation vectors shown in Table~\ref{tableperiodic}. Andreanov and Kallus \cite{ak2020} showed that no $2$-periodic sphere packing in $\R^5$ can be denser than the $D_5$ root lattice, and they obtained a complete list of the densest such packings. In the process they discovered the $Q_5$ uniform sphere packing, which also occurs in the acknowledgements of \cite{szollosi2023}. No such optimality theorem or classification is known beyond $2$-periodic packings. As for the remaining uniform packings, the $D_5$ root lattice is of course $1$-periodic, and the three $L_5$ uniform packings are $2$-periodic, $3$-periodic, and $4$-periodic (see~\cite{cs1995}).

\subsection{Symmetries of uniform packings}

We will show in this subsection that each symmetry of the $Q_5$ or $R_5$ kissing configuration extends to a symmetry of the corresponding uniform packing that fixes a point. The symmetry group of the kissing configuration is the stabilizer of a point, and the semidirect product of this group and the translational symmetry lattice has index $m$ in the symmetry group of the uniform packing when the packing is $m$-periodic (i.e., $m=2$ for $Q_5$ and $m=4$ for $R_5$).

We enumerate the symmetries as follows. Each symmetry of a subset $\mathcal{P}$ of $\R^n$ is an affine linear map $x \mapsto Ax+b$ with $A$ orthogonal. If $\Lambda = \{y \in \R^n : \mathcal{P} = \mathcal{P} + y\}$ as in the previous subsection, then $A$ must preserve $\Lambda$. To see why, note that if $x \mapsto Ax+b$ is a symmetry of $\mathcal{P}$ and $y \in \Lambda$, then
\[
\{Ax+b : x \in \mathcal{P}\} = \mathcal{P} = \mathcal{P}+y = \{A(x+y)+b :x \in \mathcal{P}\}= \{Ax+b+Ay :x \in \mathcal{P}\},
\]
from which it follows that $Ay \in \Lambda$, as desired. Thus, determining the symmetries of $\mathcal{P}$ amounts to determining the symmetries of $\Lambda$ and which translation vectors $b$ are compatible with them.

For the $Q_5$ uniform packing, the underlying lattice $\Lambda$ is the orthogonal direct sum of the $A_4$ root lattice and a one-dimensional lattice (see Table~\ref{tableperiodic} for a lattice basis). The symmetry group $G$ of $\Lambda$ is isomorphic to $S_5 \times C_2^2$, where the symmetric group $S_5$ permutes the five coordinates and the cyclic factors $C_2$ are generated by scalar multiplication by $-1$ and the reflection $s$ across the hyperplane perpendicular to the fifth lattice basis vector $v_4$ shown in Table~\ref{tableperiodic}. The reflection $s$ is the same reflection as in equation~\eqref{eqreflect}, and it corresponds to a reflection across a horizontal line in Figure~\ref{figureuniform}. Let $H$ be the subgroup of $G$ generated by $S_5$ and $s$ (the symmetry group of the kissing configuration $Q_5)$, let $H_0=H$, and let $H_1 = (-1)H$ be the coset obtained through scalar multiplication by $-1$. Then a case analysis shows that the symmetry group of the $Q_5$ uniform packing is
\[
\{x \mapsto Ax+x_i+y: A \in H_i, \ i \in \{0,1\}, \ y \in \Lambda\},
\]
where the translation vectors $x_0$ and $x_1$ are specified in Table~\ref{tableperiodic}. Each symmetry of $\Lambda$ corresponds to a unique translation vector, modulo $\Lambda$.

For the $R_5$ uniform packing, the underlying lattice $\Lambda$ is the orthogonal direct sum of the $A_3$ root lattice and two differently scaled one-dimensional lattices (again see Table~\ref{tableperiodic} for a lattice basis). The symmetry group $G$ of $\Lambda$ is isomorphic to $S_4 \times C_2^3$, where the symmetric group $S_4$ permutes the first four coordinates and the cyclic factors $C_2$ are generated by scalar multiplication by $-1$ and the reflections $r$ and $s$ across the hyperplanes perpendicular to the fourth and fifth lattice basis vectors $v_3$ and $v_4$ shown in Table~\ref{tableperiodic}. The reflections $r$ and $s$ correspond to reflections across vertical and horizontal lines in Figure~\ref{figureuniform}, respectively. Let $H$ be the subgroup of $G$ generated by $S_4$ and $s$ (the symmetry group of the kissing configuration $R_5)$, and let $H_0=H$, $H_1 = (-1)H$, $H_2 = rH$, and $H_3 = (-r) H$ be its cosets, where as above $(-1)$ denotes scalar multiplication by $-1$ and $(-r)$ denotes the composition of $(-1)$ and $r$. Then a case analysis shows that the symmetry group of the $R_5$ uniform packing is
\[
\{x \mapsto Ax+x_i+y: A \in H_i, \ i \in \{0,1,2,3\}, \ y \in \Lambda\},
\]
where again the translation vectors $x_i$ are specified in Table~\ref{tableperiodic}. As in the case of $Q_5$, each symmetry of $\Lambda$ corresponds to a unique translation vector, modulo $\Lambda$.
  
\subsection{Higher dimensions}
The four-coloring construction of five-dimensional packings from Section~\ref{section4.1} works equally well to produce six-dimensional sphere packings from valid four-colored point configurations in $\R^3$. One can construct such configurations from face-to-face tilings of $\R^3$ with irregular tetrahedra and octahedra that have $\Delta_1$ faces, as shown in Figure~\ref{figure6d}. These polyhedra have many lovely properties. For example, one octahedron and four tetrahedra can be assembled to form a larger tetrahedron (with doubled edge lengths, and with the tetrahedra nestled in the corners of the larger tetrahedron), while six octahedra and eight tetrahedra can form a larger octahedron (again with doubled edge lengths and the octahedra nestled in the corners). Iterating this procedure yields a tiling of all of space, which can be colored to produce the $E_6$ packing. Furthermore, other tilings yield all the six-dimensional sphere packings from Section~6 of \cite{cs1995}. However, we do not obtain any new six-dimensional sphere packings. Somehow the five-dimensional case has additional flexibility, which we have been unable to replicate in six or seven dimensions.

\begin{figure}
\centering
\begin{tikzpicture}
\draw (0,0)--(1,0);
\draw (-0.5,1)--(1.5,1);
\draw[dashed] (0,0)--(0.5,1)--(1,0);
\draw[dashed] (-0.5,1)--(0.5,-1)--(1.5,1);
\begin{scope}[shift={(6,{-17/100*sqrt(5)})}]
\draw ({-12/25*sqrt(5)},{16/25*sqrt(5)})--({-2/sqrt(5)},{1/sqrt(5)})--(0,0)--({2/sqrt(5)},{1/sqrt(5)})--({12/25*sqrt(5)},{16/25*sqrt(5)});
\draw[dashed] ({-12/25*sqrt(5)},{16/25*sqrt(5)})--(0,{sqrt(5)/2});
\draw[dashed] ({-2/sqrt(5)},{1/sqrt(5)})--(0,{sqrt(5)/2});
\draw[dashed] (0,0)--(0,{sqrt(5)/2});
\draw[dashed] ({2/sqrt(5)},{1/sqrt(5)})--(0,{sqrt(5)/2});
\draw[dashed] ({12/25*sqrt(5)},{16/25*sqrt(5)})--(0,{sqrt(5)/2});
\draw[dashed] ({-12/25*sqrt(5)},{16/25*sqrt(5)})--({-9/10*sqrt(5)}, {17/50*sqrt(5)})--({-2/sqrt(5)},{1/sqrt(5)})--({-2/5*sqrt(5)}, {-3/10*sqrt(5)})--(0,0)--({2/5*sqrt(5)}, {-3/10*sqrt(5)})--({2/sqrt(5)},{1/sqrt(5)})--({9/10*sqrt(5)}, {17/50*sqrt(5)})--({12/25*sqrt(5)},{16/25*sqrt(5)});
\end{scope}
\end{tikzpicture}
\caption{The irregular tetrahedron and octahedron with $\Delta_1$ facets, unfolded into nets.}
\label{figure6d}
\end{figure}
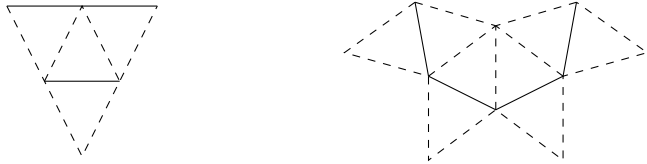

While the four-coloring construction reproduces the best packings known in up to eight dimensions, it does not work well in higher dimensions. For example, the Leech lattice in $\R^{24}$ cannot be obtained in this way. Given a valid four-colored point configuration in $\R^{21}$ with $\delta$ points per unit volume, the density of the resulting sphere packing in $\R^{24}$ is
\[
\frac{\pi^{12}}{12! \cdot 2^{13}} \, \delta.
\]
Matching the density of the Leech lattice would amount to achieving $\delta = 2^{13} = 8192$, which turns out to be impossible because each of the four colors gives a sphere packing with radius $\sqrt{2}/2$, and achieving $\delta=8192$ would require impossibly dense sphere packings. Specifically, the Blichfeldt bound 
implies that each color has at most $788.5785$ points per unit volume (see, for example, Theorem~6.1 in \cite{zong1999}), and four times that bound is still far below $8192$.

\section{Nine-dimensional kissing configurations}
\label{sec:9dkissing}

Nine dimensions is the first case beyond five dimensions in which we have been able to produce a new kissing configuration by modifying layers. We do not obtain a corresponding dense sphere packing, and indeed the densest known nine-dimensional sphere packings all have suboptimal kissing configurations.

The record kissing configuration of size $306$ in $\R^9$ was discovered by Leech and Sloane \cite{ls1971} in 1971, and it has been the only such configuration known since then \cite{cjkt2011}, up to isometry. It is obtained from a binary error-correcting code of block length~$9$, constant weight~$4$, and minimal distance~$4$. Such a code is unique~\cite{ostergard2010}, and we can construct it as follows.

It is convenient to write the nine coordinates in three rows of three. Then the code consists of $18$ codewords, obtained by arbitrarily permuting the rows and columns of
\[
\begin{bmatrix}
0 & 0 & 0\\
0 & 1 & 1\\
0 & 1 & 1
\end{bmatrix}
\qquad\text{and}\qquad
\begin{bmatrix}
0 & 1 & 1\\
1 & 0 & 0\\
1 & 0 & 0
\end{bmatrix}.
\]
The kissing configuration contains $2^4 \cdot 18 = 288$ points that place arbitrary signs on the entries of the codewords, together with the $18$ permutations of $(\pm 2, 0, \dots, 0)$, for a total of $306$ points. A simple case analysis shows that no two distinct codewords overlap in more than two nonzero coordinates, which is exactly what is needed to obtain a valid kissing configuration. 

Alternatively, we can construct the code using the finite field $\F_9$, with the nine coordinates indexed by the elements of $\F_9$. The codewords correspond to certain four-element subsets of $\F_9$, specifically the nonzero squares in $\F_9$, the non-squares in $\F_9$, and the translates of these sets by elements of $\F_9$. To see that this code is equivalent to the code from the previous paragraph, we can write $\F_9$ as $\F_3(i)$ with $i^2=-1$ and arrange the elements in a grid as
\[
\begin{bmatrix}
0 & 1 & 2\\
i & 1+i & 2+i\\
2i & 1+2i & 2+2i
\end{bmatrix};
\]
because $2 = i^2$ and $i = (2+i)^2$, the elements $1$, $2$, $i$, and $2i$ are the nonzero squares, and we obtain the same code as before. The construction using $\F_9$ is more abstract, but it lets us see that the code is invariant under the affine semilinear group, which consists of the functions $x \mapsto ax+b$ and $x \mapsto ax^3 + b$ with $a,b \in \F_9$ and $a \ne 0$. In fact, this group is the full automorphism group of the code. The automorphism group acts transitively on the codewords, since multiplying by a non-square interchanges the nonzero squares with the non-squares. By contrast, the subgroup generated by permuting the rows and columns and applying symmetries of the $3 \times 3$ square does not act transitively.

We will use the more concrete construction based on $3 \times 3$ squares of coordinates. In our calculations below, we will number the coordinates as
\[
\begin{bmatrix}
1 & 2 & 3\\
4 & 5 & 6\\
7 & 8 & 9
\end{bmatrix}.
\]
The binary code then contains the codewords shown in Table~\ref{table9dcode}.

\begin{table}
\caption{The binary code of block length~$9$, weight~$4$, and minimal distance~$4$.}\label{table9dcode}
\centering
\begin{tabular}{ccccccccc}
\toprule
1 & 1 & 0 & 1 & 1 & 0 & 0 & 0 & 0\\
1 & 1 & 0 & 0 & 0 & 1 & 0 & 0 & 1\\
1 & 1 & 0 & 0 & 0 & 0 & 1 & 1 & 0\\
1 & 0 & 1 & 1 & 0 & 1 & 0 & 0 & 0\\
1 & 0 & 1 & 0 & 1 & 0 & 0 & 1 & 0\\
1 & 0 & 1 & 0 & 0 & 0 & 1 & 0 & 1\\
1 & 0 & 0 & 1 & 0 & 0 & 0 & 1 & 1\\
1 & 0 & 0 & 0 & 1 & 1 & 1 & 0 & 0\\
0 & 1 & 1 & 1 & 0 & 0 & 1 & 0 & 0\\
0 & 1 & 1 & 0 & 1 & 1 & 0 & 0 & 0\\
0 & 1 & 1 & 0 & 0 & 0 & 0 & 1 & 1\\
0 & 1 & 0 & 1 & 0 & 1 & 0 & 1 & 0\\
0 & 1 & 0 & 0 & 1 & 0 & 1 & 0 & 1\\
0 & 0 & 1 & 1 & 1 & 0 & 0 & 0 & 1\\
0 & 0 & 1 & 0 & 0 & 1 & 1 & 1 & 0\\
0 & 0 & 0 & 1 & 1 & 0 & 1 & 1 & 0\\
0 & 0 & 0 & 1 & 0 & 1 & 1 & 0 & 1\\
0 & 0 & 0 & 0 & 1 & 1 & 0 & 1 & 1\\
\bottomrule
\end{tabular}
\end{table}

If we partition the kissing configuration into layers based on the first coordinate, we obtain layers of sizes $1$, $64$, $176$, $64$, and $1$ at heights $2$, $1$, $0$, $-1$, and $-2$, respectively. To create a new kissing configuration, we replace the layer that has first coordinate~$1$ with a modification. This mutation of an eight-dimensional layer is similar to how $L_5$, $Q_5$, and $R_5$ were constructed by modifying one four-dimensional layer of another five-dimensional kissing configuration, or to Figure~\ref{figure3d}. 

Each point in the eight-dimensional layer to be modified has a $1$ in the first coordinate and three $\pm 1$'s in one of eight triples from among the remaining coordinates. One can check that there are exactly eight other triples that could be used instead while keeping the other layers unchanged, and we define the new kissing configuration by using them in this layer.\footnote{Among the sixteen triples that are consistent with the central cross section, one can check that the original eight and their complement are the only sets of eight that avoid overlap in more than one coordinate, which would lead to overly large inner products.} Another description is that in the modified layer, we swap coordinate~$2$ with coordinate~$3$ and coordinate~$4$ with coordinate~$7$. The modified points with first coordinate~$1$ are therefore based on Table~\ref{table9dmodified}, rather than Table~\ref{table9dcode}, by putting arbitrary signs on every coordinate except the first.\footnote{Coordinates for both kissing configurations are available through DSpace@MIT at \url{https://hdl.handle.net/1721.1/157699}.}

\begin{table}
\caption{The modified codewords.}\label{table9dmodified}
\centering
\begin{tabular}{ccccccccc}
\toprule
1 & 0 & 1 & 0 & 1 & 0 & 1 & 0 & 0\\
1 & 0 & 1 & 0 & 0 & 1 & 0 & 0 & 1\\
1 & 0 & 1 & 1 & 0 & 0 & 0 & 1 & 0\\
1 & 1 & 0 & 0 & 0 & 1 & 1 & 0 & 0\\
1 & 1 & 0 & 0 & 1 & 0 & 0 & 1 & 0\\
1 & 1 & 0 & 1 & 0 & 0 & 0 & 0 & 1\\
1 & 0 & 0 & 0 & 0 & 0 & 1 & 1 & 1\\
1 & 0 & 0 & 1 & 1 & 1 & 0 & 0 & 0\\
\bottomrule
\end{tabular}
\end{table}

\begin{table}
\caption{The counts of inner products in the previously known kissing configuration (first row) and the new configuration (second row), as in Table~\ref{5dinnerproducts}.}
\label{9dinnerproducts}
\centering
\begin{tabular}{ccccccc}
\toprule
$-1$ & $-3/4$ & $-1/2$ & $-1/4$ & $0$ & $1/4$ & $1/2$\\
\midrule
$153$ & $0$ & $8640$ & $4608$ & $20016$ & $4608$ & $8640$ \\
$89$ & $384$ & $7680$ & $5888$ & $19056$ & $4992$ & $8576$\\
\bottomrule
\end{tabular}
\end{table}

The resulting kissing configuration is not isometric to the original, because we have broken the antipodal symmetry, and we obtain the following theorem: 

\begin{theorem}
There are at least two non-isometric kissing configurations of $306$ points in nine dimensions.
\end{theorem}

More generally, Table~\ref{9dinnerproducts} compares the inner product counts for the two configurations. Note that there are fewer pairs of distinct points at the minimal distance (equivalently, maximal inner product) in the modified configuration. We can interpret this fact as follows. Suppose we wish to arrange $306$ points $v_1,\dots,v_{306}$ in $\R^9$ so that they are as far from each other as possible while remaining confined to a sphere. One method is to minimize the Riesz energy
\[
\sum_{1 \le i < j \le 306} \frac{1}{|v_i-v_j|^s}
\]
for large $s$; the global optimum for energy will achieve the largest possible minimal distance in the limit as $s \to \infty$. Because our modified nine-dimensional kissing configuration has fewer pairs of points at the minimal distance, it has lower energy than the original for all sufficiently large $s$. Similarly, Table~\ref{5dinnerproducts} shows that $L_5$ has lower energy than $D_5$, $Q_5$, or $R_5$ when $s$ is large, although that comparison requires examining non-minimal distances.

Another way to distinguish these configurations is their symmetry groups. The original kissing configuration has $73728$ symmetries, namely the semidirect product of the $144$ coordinate permutations preserving the binary code with $2^9$ sign changes, while the modified configuration has only $8192$ symmetries. The group of order $8192$ is generated by the $16$ coordinate permutations (among the $144$) that fix the first coordinate, the $2^8$ sign changes that fix the first coordinate, and the map
\[
(x_1,\dots,x_9) \mapsto (-x_1,x_3,x_2,x_7,x_5,x_6,x_4,x_8,x_9)
\]
that changes the sign of the first coordinate while swapping coordinate~$2$ with coordinate~$3$ and coordinate~$4$ with coordinate~$7$.

\section*{Acknowledgements}

We are grateful to Yoav Kallus for pointing out that the kissing configuration identified by Sz\"oll\H{o}si in \cite{szollosi2023} had previously occurred as the kissing configuration of a packing discovered in \cite{ak2020}, and to the referees for helpful comments.

\end{document}